\documentclass[11pt]{article}

\usepackage[utf8]{inputenc}
\usepackage{amsthm,amsmath,amssymb,bm}
\usepackage{graphicx,tikz}
\usepackage[dvipsnames]{xcolor}
\usepackage{verbatim}
\usepackage{float}
\usepackage{mathrsfs}
\usepackage{hyperref}

\usepackage{caption}
\usepackage{array}   
\newcolumntype{C}{>{$}c<{$}}
\newcolumntype{L}{>{$}l<{$}}

\usepackage[short labels]{enumitem}
\setlist[enumerate]{topsep=0pt,itemsep=-1ex,partopsep=1ex,parsep=1ex}
\setlist[itemize]{topsep=0pt,itemsep=-1ex,partopsep=1ex,parsep=1ex}

\usepackage{bbm}

\theoremstyle{plain}
\newtheorem{theo}{Theorem}[section]
\newtheorem{prop}[theo]{Proposition}
\newtheorem{lemma}[theo]{Lemma}

\newtheorem{conj}[theo]{Conjecture}
\newtheorem{claim}[theo]{Claim}

\theoremstyle{definition}

\newcommand{\eps}{\varepsilon}

\newcommand{\bZ}{\mathbb{Z}}
\newcommand{\cM}{\mathcal{M}}
\newcommand{\cR}{\mathcal{R}}

\DeclareMathOperator{\RN}{RN_{\nu,G}(S)}

\newenvironment{proofclaim}[1][Proof of claim]{\begin{proof}[#1]}{\end{proof}}

\title{The perfect 1-factorisation conjecture holds asymptotically}
\author{Yangyang Cheng\thanks{Fakultät für Informatik und Mathematik, Universität Passau, Germany.
\emph{Email}: \href{mailto:yangyang.cheng@uni-passau.de}{\tt yangyang.cheng@uni-passau.de}, \href{mailto:amedeo.sgueglia@uni-passau.de}{\tt amedeo.sgueglia@uni-passau.de}.
YC is partially supported by
the Deutsche Forschungsgemeinschaft (DFG, German Research Foundation)-542321564;
AS is funded by the Alexander von Humboldt Foundation.}
\and Amedeo Sgueglia\footnotemark[1] 
}

\begin{document}

\maketitle

\begin{abstract}
    A famous conjecture of Anton Kotzig states that for every even integer $n\ge 4$, the complete graph $K_n$ of order $n$ can be decomposed into $n - 1$ perfect matchings such that every pair of these matchings forms a Hamilton cycle. Despite the great interest, the conjecture is far from being solved.
    Here we show that the conjecture holds asymptotically, namely that $K_n$ can be decomposed into $n-1$ perfect matchings such that $(1-o(1))n$ of them have the property that any pair forms a Hamilton cycle.
\end{abstract}

\section{Introduction}
\label{sec:intro}
A \emph{perfect matching} or \emph{$1$-factor} of a graph $G$ is a collection of pairwise vertex-disjoint edges which span the entire vertex set.
A \emph{$1$-factorisation} of $G$ is a partition of its edge set into perfect matchings.
Observe that the union of two perfect matchings (on the same vertex set) is a disjoint collection of cycles and double edges.
When the union is a Hamilton cycle, we say that the two matchings form a \emph{perfect} pair. 
In 1963 Kotzig~\cite{kotzig:64} made the following conjecture, which has become known as the \emph{perfect $1$-factorisation conjecture}.
\begin{conj}
\label{conj:kotzig}
    For every even integer $n \ge 4$, the complete graph $K_n$ on $n$ vertices admits a $1$-factorisation where any pair of perfect matchings is perfect.
\end{conj}

Kotzig~\cite{kotzig:64} observed that the following folklore $1$-factorisation of $K_n$ proves Conjecture~\ref{conj:kotzig} when $n-1$ is a prime number:
Arrange $n-1$ vertices to create the corners of a regular
$(n-1)$-gon and put the last vertex in the centre.
Denote the vertex in the centre by $u$ and choose any vertex at the corner, say $v$.
Form the first perfect matching by taking the edge $uv$ together with all the edges which are perpendicular to it.
A $1$-factorisation can then be obtained by considering the $n-1$ rotations of this matching around the centre (see~Figure~\ref{fig:perfect} where $n=6$).
One can show that if $n-1$ is prime, then any pair of matchings is perfect, which proves Conjecture~\ref{conj:kotzig} under this assumption.
Later Anderson~\cite{anderson:73} provided another infinite family of constructions which proves Conjecture~\ref{conj:kotzig} when $n/2$ is a prime number.
Moreover, the conjecture has been verified for small values of $n$ (with the current smallest open case being $n=64$) and a few more sporadic ones.
We refer to the survey of Rosa~\cite{rosa:2019} for a comprehensive list of results.
The conjecture has also connections to the acyclic edge-chromatic number, to Latin squares and to orthogonal 1-factorisations, and Glock and the second author~\cite{GS:2025} have recently considered a version of it for random graphs.

\begin{figure}[ht]
    \centering
    \includegraphics[scale=0.4]{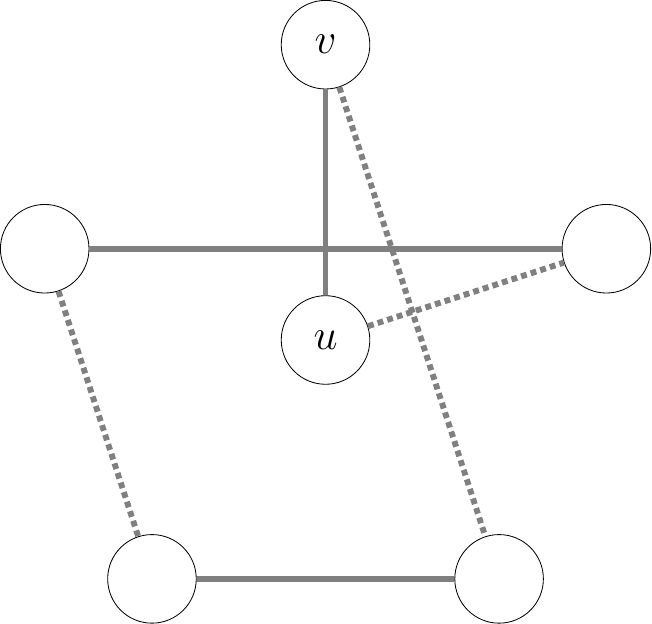}
    \caption{The solid edges form the first perfect matching; the dotted edges represent the perfect matching obtained by rotating the first one once around the centre. After rotating three more times, we get a perfect $1$-factorisation of $K_6$.}
    \label{fig:perfect}
\end{figure}

Given the difficulty of proving Conjecture~\ref{conj:kotzig} for all $n$, weaker versions have also been considered. 
Namely, one could ask for a $1$-factorisation of $K_n$ where only certain pairs of matchings are required to be perfect.
Kr{\'a}lovi{\v c} and Kr{\'a}lovi{\v c}~\cite{KK:05} constructed a $1$-factorisation $\{M_1,\dots,M_{n-1}\}$ of $K_n$ where the pair $\{M_1,M_i\}$ is perfect for each $i \in \{2,\dots,n-1\}$.
Moreover, Wagner~\cite{wagner:92} observed that in the construction of Figure~\ref{fig:perfect}, out of the $\binom{n-1}{2}$ pairs of perfect matchings, $(n-1) \cdot \frac{\varphi(n-1)}{2}$ of them are perfect, where $\varphi(\cdot)$ is the Euler's totient function.
(Observe that if $n-1$ is a prime then $\varphi(n-1)=n-2$ and $(n-1) \cdot \frac{\varphi(n-1)}{2}=\binom{n-1}{2}$, which recovers the fact that the construction is a perfect $1$-factorisation.)

Here we show that Conjecture~\ref{conj:kotzig} holds asymptotically.

\begin{theo}
\label{theo:main}
    For all $\eps>0$ there exists $n_0 \in \mathbb{N}$ such that for all even $n \ge n_0$, the complete graph $K_n$ has a $1$-factorisation which contains $(1-\eps)n$ perfect matchings such that any two of them form a perfect pair. 
\end{theo}

Theorem~\ref{theo:main} is proved via a probabilistic argument and the desired $1$-factorisation is obtained from a random alteration of a large collection of perfect matchings with the property that any pair is perfect.
For the latter result, we can in fact give a constructive proof.

\begin{theo}
\label{theo:many_matchings}
    For all $\eps>0$ there exists $n_0 \in \mathbb{N}$ such that for all even $n \ge n_0$, the complete graph $K_n$ contains $(1-\eps)n$ edge-disjoint perfect matchings such that any two of them form a perfect pair. 
\end{theo}

We remark that Theorem~\ref{theo:many_matchings} does not immediately imply Theorem~\ref{theo:main} since the construction of Theorem~\ref{theo:many_matchings} cannot be extended to a $1$-factorisation by simply adding perfect matchings; in fact we cannot even add a single perfect matching.
Instead, we show that, after deleting from $K_n$ a random large subset of the matchings given by Theorem~\ref{theo:many_matchings}, the leftover graph is a regular robust expander, whose degree is linear in $n$, and thus it contains a $1$-factorisation by a celebrated result of K\"uhn and Osthus~\cite{KO:13}.
\medskip

\noindent {\bf Comment on the use of AI.} The construction in Theorem~\ref{theo:many_matchings} is a generalisation of the one presented at the beginning of the Introduction and was generated by ChatGPT 5.4, which also gave a correct but rather long proof.
The one presented here is a cleaner, shorter and substantially different one, and was obtained by the authors.
The proof of Theorem~\ref{theo:main} was entirely obtained by the authors.

\section{Proof of Theorem~\ref{theo:many_matchings}}
\label{sec:many_matchings}
We start by informally describing the construction in Theorem~\ref{theo:many_matchings} and we refer the reader to its proof for a formal definition.
It is a generalisation of the construction described in Section~\ref{sec:intro} and Figure~\ref{fig:perfect}.
Let $r \in \mathbb{N}$ be an integer with $1 \le r \le n/2$. 
Put $n-r$ vertices as the corners of a regular polygon and the other $r$ in the centre.
Consider the matching obtained by joining the $r$ vertices in the centre to $r$ consecutive vertices on the polygon, together with all the edges across them, symmetrically on both sides.
(This can be better visualised in Figure~\ref{fig:many_matchings} drawn for $r=5$ and $n=16$.)
Then rotate the matching around the centre $n-r$ times. 
We will show that if $n-r$ is an odd prime number, then any two such matchings form a perfect pair.

\begin{prop}
\label{prop:prime}
    Let $n,r \in \mathbb{N}$ such that $n$ is even, $1 \le r \le n/2$ and $n-r$ is an odd prime number.
    Then $K_n$ contains $n-r$ perfect matchings such that any pair is perfect. 
\end{prop}

\begin{figure}[h]
    \centering
    \includegraphics[scale=0.4]{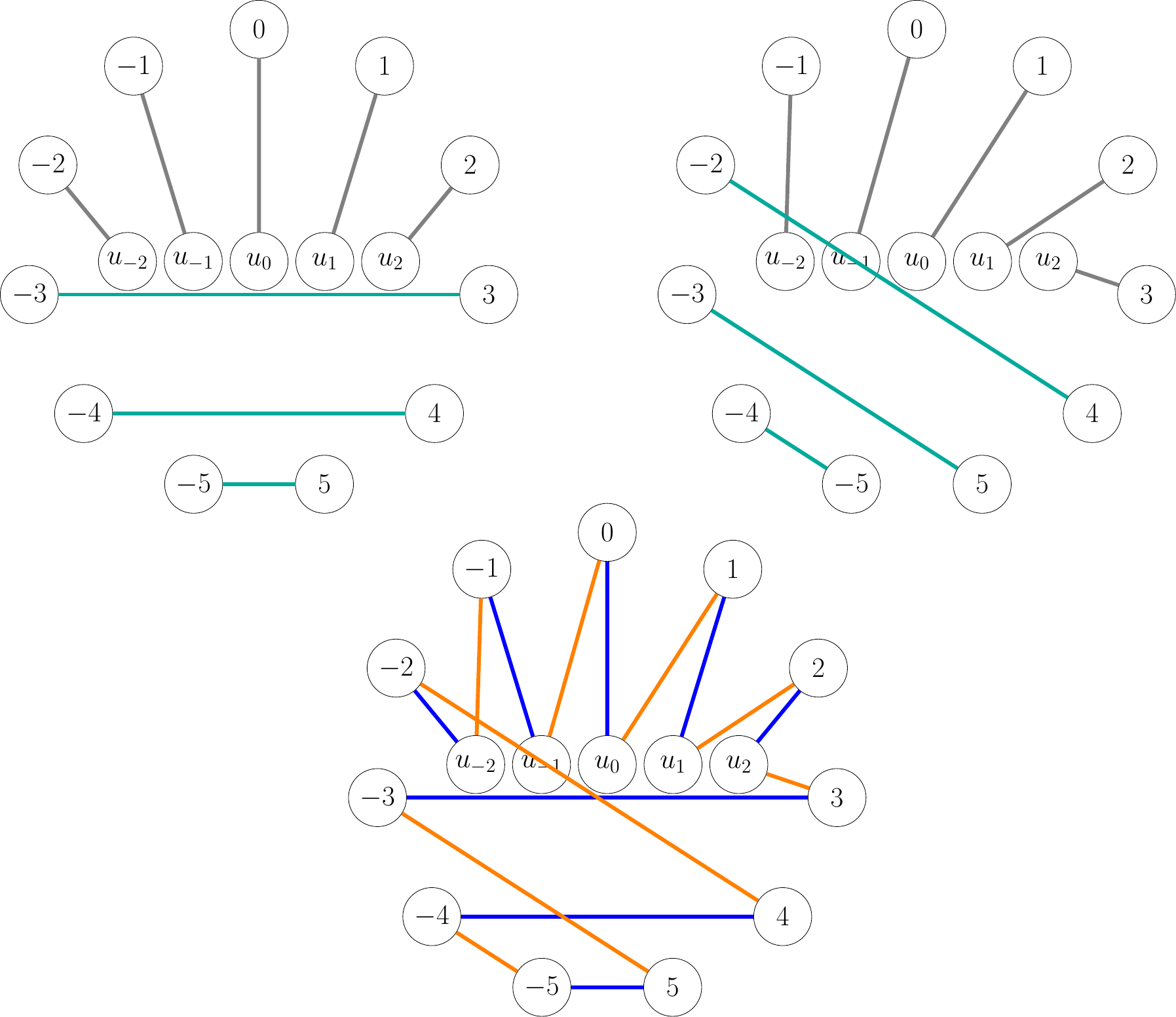}
    \caption{The figure describes the construction we use in Proposition~\ref{prop:prime} when $n=16$ and $r=5$.
    The top left graph shows the first perfect matching (with reference to the notation introduced in the proof, it is $M_0=A_0 \cup B_0$ with the gray edges belonging to $A_0$ and the green ones to $B_0$).
    The top right graph shows the matching obtained after the first rotation around the centre (so $M_1$ with, again, gray edges belonging to $A_1$ and green ones to $B_1$).
    The bottom shows that the union of the two matchings above is indeed a Hamilton cycle.}
    \label{fig:many_matchings}
\end{figure}

If we want to use Proposition~\ref{prop:prime} to prove Theorem~\ref{theo:many_matchings}, we need to show that there is a prime number $p$ sufficiently close to $n$ (and then set $r:=n-p$ in the proposition).
This problem goes back at least to a question of Legendre, who famously conjectured that there is a prime number between $n^2$ and $(n+1)^2$ for each integer $n \ge 1$.
This would follow from the existence of a prime number in each interval of the form $[n-n^\gamma,n]$ with $\gamma=0.5$.
While this is still open, many authors have found smaller and smaller $\gamma$ such that $[n-n^\gamma,n]$ always contains a prime number.
In particular, we are going to use the following result, although any $\gamma$ would suffice for us.

\begin{lemma}[Baker, Harman and Pintz~\cite{BHP:2001}]
\label{lemma:prime}
    There exists a constant $n_0>0$ such that for every integer $n\geq n_0$ the interval $[n-n^{0.525},n]$ contains a prime number.
\end{lemma}

The proof of Theorem~\ref{theo:many_matchings} readily follows.

\begin{proof}[Proof of Theorem~\ref{theo:many_matchings}]
    Let $1/n_0 \ll \eps$ and $n \ge n_0$.
    By Lemma~\ref{lemma:prime} there exists an odd prime number $p$ in the interval $[n-n^{0.525},n]$.
    Then, by Proposition~\ref{prop:prime} applied with $n$ and $r:=n-p$, we get a collection of $n-r$ perfect matchings such that any pair is perfect.
    This is enough to conclude since $n-r=p \ge n-n^{0.525} \ge (1-\eps)n$, where the last inequality uses that $n \ge n_0$.
\end{proof}

We are left to prove Proposition~\ref{prop:prime}.

\begin{proof}[Proof of Proposition~\ref{prop:prime}]
    Set $p:=n-r$ and recall that, by assumption, $p$ is an odd prime number and $r \le p$.
    Then set $q:=\frac{p-1}{2}$ and $s=\frac{r-1}{2}$ and observe they are both non-negative integers.
    Let $R:=\{u_{-s}, u_{-s+1}, \dots, u_0,\dots,u_{s-1},u_s\}$ be a set of $r$ vertices and $S$ be a set of $p$ vertices labeled with the elements of $\bZ_p=\{-q,-q+1.\dots,0,\dots,q-1,q\}$.
    When denoting the vertices of $S$ we will always think their labels modulo $p$.

    Define the perfect matching $M_0$ as the union $A_0 \cup B_0$ with $A_0:=\{(u_j,j):-s \le j \le s\}$ and $B_0:=\{(-j,j):s+1 \le j \le q\}$.
    Then define the other matchings as rotations of $M_0$ around the set $R$. More precisely, for $d \in \bZ_p$, let $M_d:=A_d \cup B_d$ with $A_d:=\{(u_j,j+d):-s \le j \le s\}$ and $B_d:=\{(-j+d,j+d):s+1 \le j \le q\}$. Since \(r\le p\), we have \(s\le q\). Hence the vertices \(j+d\), \(-s\le j\le s\), are distinct in \(\mathbb Z_p\), and
 the edges in \(B_d\) pair precisely the remaining vertices of \(S\). Thus \(M_d\) is a perfect matching. Moreover, if \(d_1\ne d_2\), then \(M_{d_1}\cap M_{d_2}=\emptyset\).
Indeed, this is clear for the edges in \(A_{d_i}\), while an edge in
\(B_d\) has the sum of its endpoints equal to \(2d\) modulo \(p\);
since \(2\) is invertible in \(\mathbb Z_p\), such an edge determines
\(d\). 
    We refer to Figure~\ref{fig:many_matchings} for an example with $n=16$ and $r=5$.

    It remains to show that any pair of perfect matchings is perfect.
    Fix two distinct $d_1,d_2 \in \bZ_p$ and let $C$ be a cycle of $M_{d_1}\cup M_{d_2}$.
    We now show that $C$ has the same number of edges in $A_{d_1}$ as it has in $A_{d_2}$, and thus the same is true for $B_{d_1}$ and $B_{d_2}$.
    Let $R':=R \cap V(C)$ and $r':=|R'|$.
    Observe that any vertex of $R'$ is incident in $C$ to two edges, which must belong to $A_{d_1}$ and $A_{d_2}$, respectively.
    Vice versa, if $C$ has an edge of $A_{d_1}$ or $A_{d_2}$, exactly one of its endpoints must belong to $R'$.
    Therefore, we can write $C \cap A_{d_1}=\{(u_j,j+d_1):j \in R''\}$, where $R''$ is the subset of $\{-s,\dots,s\}$ such that $j \in R''$ if and only if $u_j \in R'$, and $C \cap A_{d_2}=\{(u_j,j+d_2):j \in R''\}$.
    We can also conclude that $|C \cap B_{d_1}|=|C \cap B_{d_2}|$ since $C$ alternates between $M_{d_1}$ and $M_{d_2}$ and denote by $k$ their sizes.

    Consider the sum $Z:=\sum_{v \in V(C) \setminus R'} v$.
    We can calculate $Z$ in two ways.
    We can sum over the endpoints (not in $R'$) of the edges of $M_{d_1}$ and get $Z=\sum_{j \in R''}(j+d_1)+k \cdot 2d_1=\left(\sum_{j \in R''}j\right)+r' \cdot d_1+k \cdot 2d_1$.
    Doing the same for the edges of $M_{d_2}$, we get $Z=\left(\sum_{j \in R''}j\right)+r' \cdot d_2+k \cdot 2d_2$.
    By equating the two expressions modulo $p$, we get
    \[
        (r'+2k)(d_1-d_2) \equiv 0 \pmod{p}\, .
    \]
    
     Since $d_1 \not\equiv d_2 \pmod{p}$, it follows that $r'+2k \equiv 0 \pmod{p}$. 
     Since $0 \le r' \le r$ and $0 \le k \le \frac{p-r}{2}$ and they cannot be both simultaneously zero, we have $1\leq r'+ 2k \le r+p-r=p$, implying that $r'+2k=p$, $r'=r$ and $k=\frac{p-r}{2}$. Then $E(C)=2(k+r')=p+r=n$ and thus $C$ spans the entire vertex set, meaning that it is a Hamilton cycle.    
\end{proof}

We remark that if $G$ is the graph obtained from $K_n$ by removing the perfect matchings $\{M_d: d \in \bZ_p\}$ constructed in the proof above, then $G$ does not contain any perfect matching.
This is because the set $R$ is isolated in $G$ and has odd size.

\section{Proof of Theorem~\ref{theo:main}}
We start by defining the relevant terminology for robust expander.
Let \( 0 < \nu \leq \tau < 1 \). Given any graph \( G \) on \( n \) vertices and \( S \subseteq V(G) \), the \emph{\( \nu \)-robust neighbourhood} \( \RN\) of \( S \) is the set of all those vertices \( v \in V(G)\) which have at least \( \nu n \) neighbours in \( S \). Then we say that \( G \) is a \emph{robust \( (\nu, \tau) \)-expander} if \( |\RN| \geq |S| + \nu n \) for all \( S \subseteq V(G) \) with \( \tau n \leq |S| \leq (1 - \tau)n \).
K\"uhn and Osthus~\cite{KO:13} famously proved that every regular robust expander of linear degree can be decomposed in Hamilton cycles, meaning that it contains a set of edge-disjoint Hamilton cycles which together cover all the edges of the graph.
In particular, if the order of the graph is even then each Hamilton cycle can be decomposed into two edge-disjoint perfect matchings, yielding a $1$-factorisation.

\begin{theo}[Corollary of~\cite{KO:13}]
\label{theo:expander}
    For every \( \alpha > 0 \) there exists \( \tau > 0 \) such that for every \( \nu > 0 \) there exists \( n_0 = n_0(\alpha, \nu, \tau) \) for which the following holds. Suppose that 
    \begin{itemize} 
        \item $n \ge n_0$ is an even integer;
        \item \( G \) is an \( r \)-regular graph on \( n\) vertices, where \( r \geq \alpha n \);
        \item \( G \) is a robust \( (\nu, \tau) \)-expander.
    \end{itemize}
    Then \( G \) has a $1$-factorisation.
\end{theo}

Additionally, we use a result of Alon, Pokrovskiy and Sudakov~\cite{APS:17} which shows that a random subgraph of a properly edge-coloured $K_n$ formed by the edges of a random set of colours has very good expansion properties.
We remark that for any two subsets $A$ and $B$ of a graph $G$, in the quantity $e_G(A,B)$, we count twice the edges with both endpoints in $A\cap B$ and that we write $f \gg g$ if $f/g$ tends to infinity with $n$.

\begin{theo}[Theorem 1.3 in~\cite{APS:17}, see also the remark afterwards]
\label{theo:colour_classes}
    Given a proper edge-colouring of $K_n$, let $G$ be the subgraph obtained by choosing every colour class randomly and independently with probability $p \le 1/2$. Then, with high probability, for every two subsets $A, B \subseteq [n]$ with $|A|, |B| \gg (\log n/p)^2$, we have $e_G(A, B) \ge (1 - o(1))p|A||B|$.
\end{theo}

We have all the tools to prove Theorem~\ref{theo:main}.
\begin{proof}[Proof of Theorem~\ref{theo:main}]
    We can assume that $\eps \le 1/2$ and we choose constants $\tau$ and $\nu$ with $\tau \ll \eps$ and $\nu \le \tau \eps/4$.
    Moreover we assume that $n$ is sufficiently large and even. Let $\cM$ be a collection of perfect matchings of $K_n$ such that any pair in $\cM$ is perfect and $R$ be the graph obtained by deleting from $K_n$ the edges of the perfect matchings in $\cM$.
    By Theorem~\ref{theo:many_matchings}, we can assume that $|\cM| \ge (1-\eps)n$.
    
    Consider any proper colouring of $R$ and let $\cR$ be the set of the colour classes.
    Then $\cM \cup \cR$ is the set of the colour classes of a proper edge-colouring of $K_n$.
    Let $G$ be the random subgraph obtained by keeping every colour class of $\cM \cup \cR$ randomly and independently with probability $\eps$.
    If $\cM'$ denotes the set of colour classes of $\cM$ which have been kept, by a simple application of Chernoff's bound, we have that with high probability $\eps n/2 \le |\cM'| \le 2 \eps n$.
    Moreover, by Theorem~\ref{theo:colour_classes}, with high probability for every two subsets $A,B \subseteq V(K_n)$ with $|A|,|B| \ge \tau n/2$, we have $e_G(A,B) \ge \eps \tau^2 n^2/8$.
    We claim that this implies that $G$ is a robust $(\nu,\tau)$-expander.
    In fact, suppose this is not the case and let $S \subseteq V(K_n)$ be such that $\tau n \le |S| \le (1-\tau)n$ and $|\RN| < |S|+\nu n$.
    Then $|\RN| \le (1-\tau+\nu)n \le (1-\tau/2)n$ and thus there exists a set $T \subseteq V(K_n) \setminus \RN$ with $|T| = \tau n/2$.
    Then, by the above property, $e_G(S,T) \ge \eps \tau^2 n^2/8$ and, by averaging over the vertices of $T$, there exists a vertex $t \in T$ whose number of neighbours in $S$ is at least $\frac{\eps \tau^2 n^2/8}{|T|} = \eps \tau n/4 \ge \nu n$ contradicting that $t \not\in \RN$.
    
    Consider any outcome of the random choice such that $\eps n/2 \le |\cM'| \le 2 \eps n$ and $G$ is a robust $(\nu,\tau)$-expander.
    Let $H$ be the subgraph of $K_n$ obtained as the union of the edges in $\cM'$ and those in $\cR$.
    Then $H$ is $r$-regular with $r=|\cM'|+(n-1- |\cM|) \ge \eps n/2$.
    Moreover, since $H \supseteq G$, then $H$ is a robust $(\nu,\tau)$-expander as well and Theorem~\ref{theo:expander} (applied with $\alpha=\eps/2$) implies that $H$ has a $1$-factorisation $\cM''$.

    We conclude that $(\cM \setminus \cM') \cup \cM''$ is a $1$-factorisation of $K_n$. Moreover each pair of matchings in $\cM \setminus \cM'$ is perfect by construction and $|\cM \setminus \cM'| \ge (1-\eps)n-2\eps n \ge (1-3\eps)n$.
    By starting the proof with $\eps/3$ in lieu of $\eps$, we get the desired $1$-factorisation.
\end{proof}

\bibliographystyle{plain}  
\bibliography{references}  

\end{document}